\documentclass[11pt]{article}

\usepackage[margin=1in]{geometry}
\usepackage{amsmath,amssymb,amsthm}
\usepackage{enumitem}
\usepackage{microtype}
\usepackage[T1]{fontenc}
\usepackage{lmodern}
\usepackage[colorlinks=true,linkcolor=blue,citecolor=blue,urlcolor=blue]{hyperref}
\hypersetup{
  pdftitle={The List Edge-Coloring Conjecture for New Infinite Families},
  pdfauthor={Amir Jafari},
  pdfsubject={Prime congruences for signed one-factorization counts and online list edge coloring},
  pdfkeywords={list edge coloring, edge paintability, one-factorization, Pfaffian, Alon-Tarsi method, Burnside lemma, cyclic starter}
}

\newtheorem{theorem}{Theorem}[section]
\newtheorem{lemma}[theorem]{Lemma}
\newtheorem{proposition}[theorem]{Proposition}
\newtheorem{corollary}[theorem]{Corollary}

\newcommand{\F}{\mathbb F}
\newcommand{\Pf}{\operatorname{Pf}}
\newcommand{\Fix}{\operatorname{Fix}}
\newcommand{\sgn}{\operatorname{sgn}}
\newcommand{\crs}{\operatorname{cr}}

\title{The List Edge-Coloring Conjecture for\\
New Infinite Families}
\author{Amir Jafari\\[4pt]
\small New Uzbekistan University, Tashkent, Uzbekistan\\
\small\texttt{a.jafari@newuu.uz}}
\date{}

\begin{document}
\maketitle

\begin{abstract}
The List Edge-Coloring Conjecture predicts that any graph whose edges
can be colored with $k$ colors can also be colored from arbitrary lists
of $k$ colors.  We prove its stronger online form for two new infinite
families, $K_{p-1}$ and $K_{2p}$, where $p$ is an odd prime.  For even
$n$, order the vertices of $K_n$ and draw each perfect matching as arcs
above them.  Count crossings separately within each matching, and let
$S_n$ be the number of decompositions into perfect matchings having an
even total crossing count minus the number having an odd total.  Then
\[
 S_{p-1}\equiv\left(\frac{-2}{p}\right)\pmod p,
 \qquad
 S_{2p}\equiv-p\pmod {p^2}.
\]
The two congruences are governed by the same elementary matching sum
over $\F_p$, although their proofs use the prime $p$ differently.
Their nonzero residues give the conjectured values even in the online
game.  They also treat the corresponding complete graphs with one
perfect matching removed, as well as $K_{2p}$ after deleting some, but
not all, of a natural cyclic family of $p$ disjoint perfect matchings.
\end{abstract}



\section{Introduction}

For a multigraph \(G\), let \(\chi'_{\ell}(G)\) be the least integer
\(k\) such that, whenever each edge is assigned a list of \(k\) colors,
the edges can be properly colored from their lists.  Let \(\chi'(G)\)
denote the ordinary edge chromatic index.  The List Edge-Coloring Conjecture
asserts that
\[
                         \chi'_{\ell}(G)=\chi'(G)
\]
for every multigraph \(G\).  Galvin proved this for bipartite
multigraphs \cite{Galvin}.  H\"aggkvist and Janssen proved
\(\chi'_{\ell}(K_n)\le n\) \cite{HJ}.  Since \(\chi'(K_n)=n\) for odd
\(n\), this settles the conjecture for complete graphs of odd order.
The unresolved complete graphs therefore have even order.

We prove a stronger online statement.  Give each edge \(k\) tokens.
In each round, Lister marks a nonempty set of uncolored edges and
removes one token from each; Painter then colors a matching of the
marked edges.  After Painter's move, Lister wins if some uncolored edge
has no token left.
Painter wins if every edge is eventually colored.  The least \(k\) for
which Painter has a winning strategy is the edge-paintability index
\(\chi'_{\mathrm P}(G)\).  Always
\[
              \chi'(G)\le \chi'_{\ell}(G)\le\chi'_{\mathrm P}(G).
                                                               \tag{1.1}
\]
When \(n\) is even, the edges of \(K_n\) can be partitioned into
\(n-1\) perfect matchings.  Such a partition is called a
\emph{one-factorization}; it shows that \(\chi'(K_n)=n-1\).  Our main
result is that equality holds throughout (1.1) for \(K_{p-1}\) and
\(K_{2p}\), for every odd prime \(p\).

The proof reduces the coloring problem to a signed enumeration.  Fix
an order of the vertices, draw the edges of a perfect matching \(M\) as
arcs above them, and let \(\crs(M)\) be the number of pairs of arcs whose
endpoints alternate in the vertex order.  Thus the arcs \(\{a,c\}\) and
\(\{b,d\}\) cross when \(a<b<c<d\).  Put
\[
                         \varepsilon(M)=(-1)^{\crs(M)}.
\]
If \(\mathcal F\) is a one-factorization, put
\[
              \varepsilon(\mathcal F)
                 =\prod_{M\in\mathcal F}\varepsilon(M).
\]
Although the sign of one matching may change when the vertices are
reordered, the product \(\varepsilon(\mathcal F)\) does not.  For even
\(n\), define
\[
                         S_n=\sum_{\mathcal F}\varepsilon(\mathcal F),
\]
where the sum is over all one-factorizations of \(K_n\); the factors
themselves are not ordered.  Thus \(S_n\) is simply the number of
positive factorizations minus the number of negative ones.

For \(K_n\), the relevant Alon--Tarsi coefficient is
\(\pm(n-1)!S_n\).  The coefficient method \cite{AT,EG}, together with
Schauz's online Combinatorial Nullstellensatz
\cite{SchauzPrimeDegree,SchauzOrientations}, therefore gives
\[
 S_n\ne0
 \quad\Longrightarrow\quad
 \chi'(K_n)=\chi'_{\ell}(K_n)=\chi'_{\mathrm P}(K_n)=n-1.
\]
Thus it remains to prove that the positive and negative factorizations
do not cancel.

Schauz proved the resulting online equality at order \(p+1\):
\[
                         \chi'_{\mathrm P}(K_{p+1})=p
\]
for every odd prime \(p\) \cite{SchauzPrimeDegree}.  Rabern's
enumeration also gives \(K_{10}\) \cite{Rabern}.  We obtain two new
infinite families.  Recall that the Legendre symbol
\((\frac{a}{p})\), for \(p\nmid a\), is \(1\) or \(-1\) according as
\(a\) is or is not a square modulo \(p\).

\begin{theorem}\label{thm:intro-pminus1}
Let \(p\) be an odd prime.  Then
\[
 S_{p-1}\equiv\left(\frac{-2}{p}\right)\pmod p,
\]
and
\[
 \chi'(K_{p-1})
 =\chi'_{\ell}(K_{p-1})
 =\chi'_{\mathrm P}(K_{p-1})=p-2.
\]
\end{theorem}

\begin{theorem}\label{thm:intro-2p}
For every odd prime \(p\),
\[
                         S_{2p}\equiv-p\pmod{p^2}.
\]
Consequently,
\[
 \chi'(K_{2p})
 =\chi'_{\ell}(K_{2p})
 =\chi'_{\mathrm P}(K_{2p})=2p-1.
\]
\end{theorem}

The need for two proofs has a simple source.  If the factors are
ordered, the signed sum becomes \((n-1)!S_n\).  At order \(p-1\), the
factor \((p-2)!\) is nonzero modulo \(p\), so an ordinary coefficient
calculation retains \(S_{p-1}\).  At order \(2p\), the factor
\((2p-1)!\) is divisible by \(p\), so reduction of the ordered
coefficient modulo \(p\) cannot detect \(S_{2p}\).  We instead act by
translations on two \(p\)-vertex layers and count the invariant
factorizations directly modulo \(p^2\).

The two congruences are unified by the same elementary object.  Write
\(\F_p\) for the integers modulo \(p\).  A \emph{starter} is a perfect
matching \(A\) of \(\F_p\setminus\{0\}\) such that, as \(\{a,b\}\)
runs through its pairs, the differences \(\pm(a-b)\) run through all
nonzero residues.  For example,
\(\{\{1,4\},\{2,3\}\}\) is a starter in \(\F_5\).  Order the nonzero
residues as \(1,\ldots,p-1\), give each starter its crossing sign
\(\varepsilon(A)\), and set
\[
                         T_p=\sum_{A\text{ starter}}\varepsilon(A).
\]
This signed sum satisfies
\[
                         T_p\equiv-\left(\frac2p\right)\pmod p.
\]
The same sum controls all three prime-indexed orders:
\[
 \begin{aligned}
 S_{p-1}&\equiv(-1)^{(p+1)/2}T_p &&\pmod p,\\
 S_{p+1}&\equiv(-1)^{(p-1)/2}T_p &&\pmod p,\\
 S_{2p}&\equiv-pT_p^2 &&\pmod {p^2}.
 \end{aligned}
\]

There are two useful consequences.  First, if \(n\ge4\) is even and
\(I\) is any perfect matching of \(K_n\), then
\[
 S_n\ne0\quad\Longrightarrow\quad
 \chi'(K_n-I)=\chi'_{\ell}(K_n-I)
              =\chi'_{\mathrm P}(K_n-I)=n-2.
\]
Second, label two sets of \(p\) vertices by \(\F_p\), and, for
\(d\in\F_p\), let \(M_d\) join \(x\) in the first set to \(x+d\) in
the second.  Given \(E\subset\F_p\) with \(1\le r=|E|\le p-1\), put
\(H=K_{2p}-\bigcup_{d\in E}M_d\).  Then
\[
 \chi'(H)=\chi'_{\ell}(H)=\chi'_{\mathrm P}(H)=2p-1-r.
\]

The paper is organized as follows.  Section~2 establishes the signed
coefficient and its coloring consequence.  Sections~3 and~4 compute
the counts for \(K_{p-1}\) and \(K_{2p}\), respectively.  Section~5
evaluates the common starter sum and also recovers the known result for
\(K_{p+1}\).  Section~6 proves the two deletion results just described.

\section{The signed count}

Fix the vertex order $0,1,\ldots,n-1$ on $K_n$, where $n$ is even.
For a perfect matching $M$, draw its edges as arcs above the ordered
vertices and set
\[
             \varepsilon(M)=(-1)^{\crs(M)},
\]
where $\crs(M)$ is the number of crossing pairs of arcs.  Equivalently,
for the generic skew-symmetric matrix $X=(X_{ij})$ with
\[
 X_{ij}=x_{ij},\quad X_{ji}=-x_{ij}\quad(i<j),
\]
we have
\[
                  \Pf(X)=\sum_M\varepsilon(M)x^M.
\]
Here and below,
\[
 x^M=\prod_{\{i,j\}\in M}x_{ij},
 \qquad
 x^{\mathbf{1}}=\prod_{0\le i<j<n}x_{ij}.
\]
For any polynomial $Q$, the notation $[x^\alpha]Q$ means the
coefficient of the monomial $x^\alpha$ in $Q$.
A one-factorization is a partition of the edge set into perfect
matchings.  For an unordered one-factorization $\mathcal F$ of $K_n$,
define
\[
 \varepsilon(\mathcal F)=\prod_{M\in\mathcal F}\varepsilon(M),
 \qquad
 S_n=\sum_{\mathcal F}\varepsilon(\mathcal F).
\]
More generally, if $G$ is a regular graph on the same ordered vertex
set, write $S(G)$ for the corresponding signed sum over its unordered
one-factorizations, with $S(G)=0$ when none exist.
The corresponding squarefree Pfaffian coefficient is
\[
 \Theta_n=[x^{\mathbf{1}}]\Pf(X)^{n-1}=(n-1)!\,S_n,       \tag{2.1}
\]
since a contribution to the squarefree monomial is exactly an ordered
one-factorization.  The $n-1$ perfect matchings in a factorization are
distinct, so each unordered factorization has exactly $(n-1)!$
orderings.

We shall need the following relabeling law.  Its specialization to a
complete graph says that the factorization weight is invariant.

\begin{lemma}[relabeling law]\label{lem:sign-invariance}
Let $G$ be a $k$-regular graph on an ordered vertex set, let $\pi$ be an
automorphism of $G$, let $P_\pi$ be its permutation matrix, and put
\[
 \iota_G(\pi)=
 \#\{\{i,j\}\in E(G):i<j,\ \pi(i)>\pi(j)\}.
\]
Then every one-factorization $\mathcal F$ of $G$ satisfies
\[
 \varepsilon(\pi\mathcal F)
 =\det(P_\pi)^k(-1)^{\iota_G(\pi)}\varepsilon(\mathcal F).
\]
In particular, if $G=K_n$ and $n$ is even, then
$\varepsilon(\pi\mathcal F)=\varepsilon(\mathcal F)$.
\end{lemma}

\begin{proof}
The Pfaffian identity
$\Pf(P_\pi XP_\pi^{\mathsf T})=\det(P_\pi)\Pf(X)$ gives the
transformation law for one matching.  After the variables are restored
to the convention $x_{uv}$ with $u<v$, it reads
\[
 \varepsilon(\pi M)=\det(P_\pi)(-1)^{f(\pi,M)}\varepsilon(M),
 \qquad
 f(\pi,M)=\#\{\{i,j\}\in M:i<j,\ \pi(i)>\pi(j)\}.
\]
Multiplication over the $k$ factors of $\mathcal F$ gives the stated
formula, because those factors partition $E(G)$ and therefore
$\sum_{M\in\mathcal F}f(\pi,M)=\iota_G(\pi)$.

For $G=K_n$, one has $k=n-1$ and
$(-1)^{\iota_G(\pi)}=\det(P_\pi)$.  Hence
\[
 \varepsilon(\pi\mathcal F)
 =\det(P_\pi)^n\varepsilon(\mathcal F)
 =\varepsilon(\mathcal F),
\]
because $n$ is even.
\end{proof}

We next prove explicitly that the Pfaffian sign is the sign seen by the
graph polynomial.  For a graph $H$ whose vertices have been ordered,
write
\[
                         P_H=\prod_{\substack{uv\in E(H)\\u<v}}(x_u-x_v)
\]
for its graph polynomial.

We use the following form of the online Combinatorial Nullstellensatz.
If $t_v\ge0$, if $\sum_vt_v=|E(H)|$, and if
\[
                         \left[\prod_vx_v^{t_v}\right]P_H\ne0,
\]
then $H$ is paintable when each vertex $v$ begins with $t_v+1$ tokens
\cite{SchauzPrimeDegree,SchauzOrientations}.  We apply this to the line
graph $L(G)$, whose vertices, and hence whose variables, are the edges
of $G$.

\begin{proposition}[coefficient and paintability bridge]
\label{prop:coefficient-bridge}
Let $G$ be a simple $k$-regular graph of class $1$ on an even number of
ordered vertices.  If $S(G)$ is the sum of the Pfaffian signs of its
unordered one-factorizations, then, for some constant
$\delta_G\in\{\pm1\}$ depending only on the chosen vertex and edge
orders,
\[
 \left[\prod_{e\in E(G)}x_e^{k-1}\right]P_{L(G)}
                         =\delta_G k!\,S(G).              \tag{2.2}
\]
Consequently,
\[
 S(G)\ne0\quad\Longrightarrow\quad
 \chi'_{\mathrm P}(G)=\chi'_{\ell}(G)=k.                 \tag{2.3}
\]
\end{proposition}

\begin{proof}
Fix an order of $E(G)$ and give the incident edges at every vertex the
induced order.  Put $A=\{1,\ldots,k\}$ and
\[
 N(a)=\prod_{t\in A\setminus\{a\}}(a-t),
 \qquad
 \Delta=\prod_{1\le i<j\le k}(i-j).
\]
The total degree of $P_{L(G)}$ is $|E(G)|(k-1)$.  The quantitative
Combinatorial Nullstellensatz therefore evaluates the coefficient on
the left of (2.2) as
\[
 \sum_{c\in A^{E(G)}}
 \frac{P_{L(G)}(c)}{\prod_{e\in E(G)}N(c(e))}.             \tag{2.4}
\]
Only proper $k$-edge-colorings contribute.  In such a coloring every
color occurs once at each vertex, hence each color class is a perfect
matching and contains $|V(G)|/2$ edges.

For a contributing coloring $c$, let $\rho_v\in\mathfrak S_k$ be the
permutation whose values are the colors on the incident edges at $v$,
read in their fixed order.  Since every adjacent pair of edges in a
simple graph has a
unique common endpoint,
\[
 P_{L(G)}(c)=\Delta^{|V(G)|}
             \prod_{v\in V(G)}\sgn(\rho_v).
\]
Moreover,
\[
 \prod_{e\in E(G)}N(c(e))
   =\left(\prod_{a=1}^kN(a)\right)^{|V(G)|/2}
   =\left((-1)^{\binom{k}{2}}\Delta^2\right)^{|V(G)|/2}.
\]
Thus the summand in (2.4) is a fixed sign, independent of $c$, times
$\prod_v\sgn(\rho_v)$.

It remains to identify this local sign.  Consider the set of half-edges
of $G$.  First order it by vertex and then by the fixed edge order.
Sorting the half-edges at each vertex by their colors has sign
$\prod_v\sgn(\rho_v)$.  Transposing the resulting vertex-by-color array
to color-by-vertex order has a sign independent of $c$.  Within the
block of a color $a$, reorder the vertices into consecutive endpoint
pairs of the matching $M_a=c^{-1}(a)$, writing the smaller endpoint
first in each pair.  The sign of this last
permutation is precisely $\varepsilon(M_a)$: it is the Pfaffian sign of
the ordered matching.  Finally, sort the two-element blocks globally by
the fixed edge order.  This introduces no sign, because an interchange
of two blocks of size two is even.  The initial and final half-edge
orders are now fixed.  Comparing the signs of the composite reordering
gives
\[
 \prod_v\sgn(\rho_v)=\delta'_G
                     \prod_{a=1}^k\varepsilon(M_a),       \tag{2.5}
\]
where $\delta'_G$ is independent of $c$.  This is the
Alon--Tarsi--Ellingham--Goddyn sign bridge; the same identification is
given in \cite[Theorem~3.12]{SchauzOrientations}.

Each unordered one-factorization admits exactly $k!$ assignments of
the colors in $A$, and (2.5) shows that all have the same sign.
Substitution in (2.4) proves (2.2).  If $S(G)\ne0$, the coefficient is a
nonzero integer and hence is nonzero over $\mathbb Q$.  The online
criterion above, with exponent $k-1$ at every vertex of $L(G)$, gives
$k$-paintability.  The reverse inequalities follow from
$\chi'(G)=k$ and (1.1).
\end{proof}

For $G=K_n$ we have $k=n-1$ and $S(G)=S_n$.  We shall prove $S_n\ne0$
by finding a nonzero residue modulo a suitable prime.  Since $S_n$ is
an integer, such a residue suffices.

\section{Complete graphs of order
\texorpdfstring{$p-1$}{p-1}}

Fix an odd prime $p$ and put
\[
                         N=p-1=2m.
\]
Since $(p-2)!$ is nonzero modulo $p$, equation~(2.1) shows that it is
enough to compute the ordered Pfaffian coefficient $\Theta_N$.
The calculation has three steps.  Frobenius first singles out one matching
$M_0$.  Glynn's theorem then replaces the remaining directed coefficient by
explicit factorial weights.  Finally, a divided difference restores a
deleted Vandermonde and reduces the answer to a Pfaffian supported on a
single antidiagonal.

All calculations in Sections~\ref{sec:glynn-bridge} and
\ref{sec:antidiagonal} take place over $\F_p$; we use the same notation
for an integer coefficient and its reduction modulo $p$.  We use the
consecutive matching
\[
 M_0=\bigl\{\{0,1\},\{2,3\},\ldots,\{N-2,N-1\}\bigr\},
                                                               \tag{3.1}
\]
whose arcs are pairwise noncrossing, so that $\varepsilon(M_0)=1$.

\subsection{Glynn's coefficient congruence}\label{sec:glynn-bridge}

Glynn's congruence removes the permutation expansion from a determinant
coefficient: once the line sums are fixed, only factorial weights remain.

For an $r\times r$ matrix $Y=(y_{ij})$ and a nonnegative integral matrix
$L=(\ell_{ij})$, write
\[
              (\det Y)^{p-1}=\sum_L C_LY^L,
 \qquad
 Y^L=\prod_{i,j}y_{ij}^{\ell_{ij}},
 \qquad
 L!=\prod_{i,j}\ell_{ij}!.
\]
Every monomial that occurs has all row and column sums of $L$ equal to
$p-1$, because each determinant factor contributes exactly one entry
from every row and every column.  Glynn's theorem is the following
uniform statement~\cite{Glynn}; the same coefficient form is also
recorded in~\cite{ItohShimoyoshi}.

\begin{theorem}[Glynn]\label{thm:glynn}
If every row sum and every column sum of $L$ is $p-1$, then
\[
                         L!\,C_L\equiv(-1)^r\pmod{p}.          \tag{3.2}
\]
\end{theorem}

We apply this with $r=N$, which is even.  Every relevant entry of $L$ is
at most $p-1$, so all local factorials are units and
\[
                         C_L\equiv\frac1{L!}\pmod{p}.        \tag{3.3}
\]

Frobenius isolates the Pfaffian coefficient to which we apply (3.3).
Since $\det X=\Pf(X)^2$ and $N=p-1$,
\[
 (\det X)^N=\Pf(X)^{2p-2}=\Pf(X)^p\Pf(X)^{p-2}.             \tag{3.4}
\]
Moreover, over $\F_p$,
\[
                         \Pf(X)^p=\sum_M\varepsilon(M)x^{pM}.
\]
Define undirected target exponents
\[
 a_{ij}=
 \begin{cases}
 p+1,&\{i,j\}\in M_0,\\
 1,&\{i,j\}\notin M_0.
 \end{cases}                                               \tag{3.5}
\]
Only $M_0$ can be selected from the Frobenius factor: a matching
$M\ne M_0$ contains an edge outside $M_0$, whose target exponent $1$
cannot absorb the exponent $p$.  Since $\varepsilon(M_0)=1$ and the
residual exponents $a_{ij}-p\,[\,\{i,j\}\in M_0\,]$ are all equal to
$1$, equation (3.4) gives
\[
 \Theta_N=\left[\prod_{i<j}x_{ij}^{a_{ij}}\right](\det X)^N.
                                                               \tag{3.6}
\]

We first apply (3.3) to $(\det Y)^N$, and only afterward impose the skew
specialization
\[
 y_{ij}=x_{ij},\qquad y_{ji}=-x_{ij}\quad(i<j),\qquad y_{ii}=0.
\]
Fix an edge $\{i,j\}$ whose target exponent is $a$.  A directed lift uses
$y_{ij}$ with exponent $s$ and $y_{ji}$ with exponent $t$, where $s+t=a$;
after the skew specialization its local weight is $(-1)^t/(s!t!)$.  Thus all
directed lifts of this one edge are recorded by
\[
 G_a(u,v)=
 \sum_{\substack{s+t=a\\0\le s,t\le N}}
             \frac{(-1)^t u^sv^t}{s!\,t!}.                  \tag{3.7}
\]
Diagonal terms vanish under the specialization.  To impose the row and
column sums, introduce markers $r_0,\ldots,r_{N-1}$ and
$c_0,\ldots,c_{N-1}$.  An occurrence of $y_{ij}$ contributes $r_ic_j$;
hence the exponent of $r_i$ is the $i$th row sum and that of $c_j$ is the
$j$th column sum.  Equations (3.3) and (3.6) now give
\[
 \Theta_N=
 \left[\prod_i r_i^Nc_i^N\right]
 \prod_{0\le i<j<N}G_{a_{ij}}(r_ic_j,r_jc_i).              \tag{3.8}
\]
For an edge outside $M_0$,
\[
                         G_1(u,v)=u-v.                     \tag{3.9}
\]
For a matched edge, $a=p+1=N+2$, and the constraint $s,t\le N$ forces
$2\le s\le N$.  Writing $t=p+1-s$ and using
$(p-k)!\equiv(-1)^k/(k-1)!\pmod{p}$ (a form of Wilson's theorem), one
finds $(-1)^t/(s!\,t!)\equiv-1/\bigl(s(s-1)\bigr)$, so that
\[
 G_{p+1}(u,v)
   =-\sum_{s=2}^{N}\frac{u^sv^{N+2-s}}{s(s-1)}.            \tag{3.10}
\]

Equation (3.8) still uses two families of markers, although after the skew
specialization only one is needed.  Put $r_i=z_ic_i$.  Each occurrence on
the edge $\{i,j\}$ then contributes one factor of both $c_i$ and $c_j$.
The total target degree at every vertex is $(N+2)+(N-2)=2N$, so the powers
of the $c_i$ are forced; it remains only to extract $z_i^N$ at each vertex.

For an edge outside $M_0$, equation (3.9) now contributes $z_i-z_j$.
These factors form the Vandermonde with precisely the $M_0$-factors
omitted.  Each edge of $M_0$ contributes the same homogeneous kernel.
Accordingly, put
\[
 V_{M_0}(z)=
 \prod_{\substack{0\le i<j<N\\\{i,j\}\notin M_0}}(z_i-z_j),
 \qquad
 K(x,y)=\sum_{s=2}^{N}\frac{x^sy^{N+2-s}}{s(s-1)}.
                                                               \tag{3.11}
\]
There are $m$ negative signs in (3.10), one for each block of $M_0$,
and hence
\[
 \Theta_N=(-1)^m T,
 \qquad
 T=\left[\prod_i z_i^N\right]
      V_{M_0}(z)\prod_{\{u,v\}\in M_0}K(z_u,z_v).         \tag{3.12}
\]
At this point the determinant and Frobenius calculations are finished; it
remains to evaluate the single coefficient $T$.

\subsection{The antidiagonal Pfaffian}\label{sec:antidiagonal}

The expression $T$ contains the Vandermonde with the $m$ factors belonging
to $M_0$ deleted.  To restore the missing factor $z_u-z_v$ on a block
$\{u,v\}$, we move one unit of exponent to $u$ or to $v$.  We therefore seek
a kernel whose discrete difference
\[
                  \kappa(a+1,b)-\kappa(a,b+1)
\]
is the complementary coefficient of $K$.  The complementary monomial is
$x^{N-a}y^{N-b}$.  Since $K$ has degree $N+2$, its coefficient can be
nonzero only when $a+b=N-2$.  The full Vandermonde is
\[
                         V(z)=\prod_{i<j}(z_i-z_j).         \tag{3.13}
\]

The following kernel has exactly the required property.
For $0\le\alpha,\beta\le N-1$, define
\[
 \kappa(\alpha,\beta)=
 \begin{cases}
 -1/(\alpha+1),&\alpha+\beta=p-2,\\
 0,&\text{otherwise}.
 \end{cases}                                               \tag{3.14}
\]

\begin{lemma}[divided difference]\label{lem:divided-difference}
Writing the blocks of $M_0$ as $\{u_i,v_i\}$ with $u_i<v_i$,
\[
 T=\sum_{\gamma\in\{0,\ldots,N-1\}^N}
       [z^\gamma]V(z)\prod_{i=1}^m
       \kappa(\gamma_{u_i},\gamma_{v_i}).                 \tag{3.15}
\]
\end{lemma}

\begin{proof}
Write $c(\delta)=[z^\delta]V_{M_0}$.  On the left side of (3.15), the
coefficient of $x^Ay^B$ in $K$ is $1/\bigl(A(A-1)\bigr)$ when $A+B=p+1$
and $2\le A,B\le N$, and zero otherwise; extracting
$\prod_iz_i^N$ therefore pairs $c(\delta)$ with the kernel coefficients
at $(A,B)=(N-\delta_u,\,N-\delta_v)$, so only exponent vectors with
\[
                         \delta_u+\delta_v=N-2             \tag{3.16}
\]
on each block occur.

On the right side, use $V=V_{M_0}\prod_i(z_{u_i}-z_{v_i})$ and expand
the last product.  For a fixed block $\{u,v\}$, the two expansion terms
shift the exponent of $z_u$ or $z_v$ by one, so after reindexing the
$\gamma$-sum by $\delta$ the block contributes
\[
 \kappa(\delta_u+1,\delta_v)-
 \kappa(\delta_u,\delta_v+1).
\]
If $\delta_u=N-1$ (or symmetrically $\delta_v=N-1$), one displayed
argument of $\kappa$ is $N$, just outside the range in (3.14); we extend
$\kappa$ by zero there.  This causes no new term, since the other
exponent would have to be $-1$.  With this harmless convention, the
displayed difference vanishes unless (3.16) holds; on that locus it equals
\[
 -\frac1{\delta_u+2}+\frac1{\delta_u+1}
 =\frac1{(\delta_u+1)(\delta_u+2)}
 =\frac1{(N-\delta_u)(N-1-\delta_u)},                     \tag{3.17}
\]
because $N\equiv-1\pmod{p}$.  This is precisely the coefficient of
$z_u^{N-\delta_u}z_v^{N-\delta_v}$ in $K(z_u,z_v)$.  Multiplying over
the $m$ blocks and summing against $c(\delta)$ proves (3.15).
\end{proof}

The full Vandermonde has one nonzero coefficient for each permutation of
the exponents $0,\ldots,N-1$, and that coefficient is its sign, up to one
fixed global sign.  Thus (3.15) is an alternating sum over assignments of
exponent labels to the vertices.  Pairing the assignments that interchange
the two vertices of a block shows that only the antisymmetric part of
$\kappa$ can contribute.  This is precisely the Pfaffian mechanism.

Define the antisymmetric $N\times N$ matrix
\[
                         A_{\alpha\beta}
  =\frac12\bigl(\kappa(\alpha,\beta)-\kappa(\beta,\alpha)\bigr),
 \qquad 0\le\alpha,\beta<N.                               \tag{3.18}
\]

\begin{lemma}[Pfaffian form]\label{lem:pf-form}
\[
                         T=(-1)^m2^m m!\,\Pf(A).             \tag{3.19}
\]
\end{lemma}

\begin{proof}
The coefficient $[z^\gamma]V$ vanishes unless
$\gamma=(\rho(0),\ldots,\rho(N-1))$ for a permutation $\rho$ of
$\{0,1,\ldots,N-1\}$.  With the convention (3.13), the expansion
$V=\det\bigl(z_i^{N-1-j}\bigr)_{i,j}$ shows that the coefficient of the
decreasing assignment $\gamma_i=N-1-i$ is $+1$; passing from the
decreasing to the increasing enumeration of exponents costs the sign of
the order-reversing permutation, which is
$(-1)^{\binom{N}{2}}=(-1)^{m(p-2)}=(-1)^m$.  Hence
\[
                         [z^\gamma]V=(-1)^m\sgn(\rho).
\]
Substituting into (3.15),
\[
 T=(-1)^m\sum_{\rho\in\mathfrak{S}_N}\sgn(\rho)
      \prod_{i=1}^m\kappa\bigl(\rho(u_i),\rho(v_i)\bigr).
\]
We antisymmetrize one block at a time.  For a fixed $i$, partition the
permutations into pairs
$\{\rho,\rho\circ(u_i\,v_i)\}$.  The two permutations have opposite
signs, and composition with $(u_i\,v_i)$ leaves
$\rho(u_j),\rho(v_j)$ unchanged for every $j\ne i$.  Averaging within
each pair therefore replaces the $i$th kernel by its antisymmetric part
without changing the sum.  Iterating this argument over
$i=1,\ldots,m$ gives
\[
 \sum_{\rho\in\mathfrak{S}_N}\sgn(\rho)
      \prod_{i=1}^m\kappa\bigl(\rho(u_i),\rho(v_i)\bigr)
 =
 \sum_{\rho\in\mathfrak{S}_N}\sgn(\rho)
      \prod_{i=1}^m A_{\rho(u_i),\rho(v_i)}.
\]
The standard permutation expansion
\[
 \Pf(A)=\frac1{2^m m!}
 \sum_{\rho\in\mathfrak{S}_N}\sgn(\rho)
      \prod_{i=1}^m A_{\rho(u_i),\rho(v_i)}
\]
now gives (3.19).
\end{proof}

It remains to evaluate this Pfaffian.  The special choice of $\kappa$ now
pays off: $A$ is supported on one antidiagonal, so its Pfaffian has only one
nonzero matching term.

\begin{lemma}[single antidiagonal]\label{lem:single-antidiagonal}
\[
                         \Pf(A)=\frac{(-1)^m}{m!}.          \tag{3.20}
\]
\end{lemma}

\begin{proof}
If $\alpha+\beta=p-2$, then $\beta+1\equiv-(\alpha+1)\pmod{p}$, so
$\kappa(\beta,\alpha)=-1/(\beta+1)\equiv1/(\alpha+1)$ and
\[
 A_{\alpha\beta}=\frac12\left(-\frac1{\alpha+1}-\frac1{\alpha+1}\right)
               =-\frac1{\alpha+1};
\]
all other entries vanish.  Thus $A$ is supported on the single
antidiagonal $\alpha+\beta=N-1$, and the only perfect matching of
$\{0,\ldots,N-1\}$ inside this support is
\[
 \{0,N-1\},\{1,N-2\},\ldots,\{m-1,m\}.
\]
For example, when $p=5$, these pairs are $\{0,3\}$ and $\{1,2\}$.
These pairs are nested, so no two arcs cross and the Pfaffian sign of
this matching is $+1$.  Therefore
\[
 \Pf(A)=\prod_{\alpha=0}^{m-1}\left(-\frac1{\alpha+1}\right)
       =\frac{(-1)^m}{m!}.
\]
\end{proof}

Combining Lemmas~\ref{lem:pf-form} and
\ref{lem:single-antidiagonal} gives $T=2^m$.  Equation (3.12) therefore
proves the congruence
\[
                         \Theta_{p-1}=(-2)^m\quad\text{in }\F_p.
                                                               \tag{3.21}
\]

\begin{proof}[Proof of Theorem~\ref{thm:intro-pminus1}]
The residue in (3.21) is nonzero.  By (2.1),
\[
              \Theta_{p-1}=(p-2)!\,S_{p-1}.
\]
Wilson's theorem gives $(p-2)!\equiv1\pmod{p}$, so
$S_{p-1}\equiv(-2)^m\not\equiv0\pmod{p}$; in particular $S_{p-1}$ is a
nonzero integer.  The complete graph $K_{p-1}$ is one-factorable, so
$\chi'(K_{p-1})=p-2$; Proposition~\ref{prop:coefficient-bridge} gives the
remaining equalities:
\[
 \chi'(K_{p-1})=\chi'_{\ell}(K_{p-1})
                =\chi'_{\mathrm P}(K_{p-1})=p-2.
\]
Euler's criterion and the preceding Wilson congruence now give
\[
                         S_{p-1}\equiv(-2)^m
                         =\left(\frac{-2}{p}\right)\pmod p,
\]
as required.
\end{proof}

\section{Complete graphs of order
\texorpdfstring{$2p$}{2p}}

Here $(2p-1)!$ is divisible by $p$, so the ordered Pfaffian
coefficient no longer determines $S_{2p}$ modulo $p$.  We instead
evaluate the unordered signed count by a translation action on two
$p$-vertex layers.

We now fix an odd prime $p$ and write the vertices as two layers
\[
 V=V_0\sqcup V_1,
 \qquad V_i=\F_p\times\{i\},
\]
ordered with all of $V_0$ (in the order of $\F_p$) before all of $V_1$.
The group $\Gamma=\F_p^2$ acts by independent translations,
\[
 (a,b):(x,0)\mapsto(x+a,0),
 \qquad (y,1)\mapsto(y+b,1).                              \tag{4.1}
\]
By Lemma~\ref{lem:sign-invariance}, the weight
$\varepsilon(\mathcal F)$ is constant on $\Gamma$-orbits.

We use the signed form of Burnside's lemma.

\begin{lemma}[weighted Burnside]\label{lem:weighted-burnside}
Let a finite group $G$ act on a finite set $\mathcal X$, and suppose
$w:\mathcal X\to\{\pm1\}$ is constant on orbits.  Set
\[
                         W(g)=\sum_{x\in\Fix(g)}w(x).
\]
Then
\[
 \sum_{g\in G}W(g)=|G|\sum_{\mathcal O\in\mathcal X/G}w(\mathcal O).
                                                               \tag{4.2}
\]
In particular the left side is divisible by $|G|$.
\end{lemma}

\begin{proof}
For an orbit $\mathcal O$, the number of pairs $(g,x)$ with
$x\in\mathcal O$ and $gx=x$ is
$|\mathcal O|\,|\operatorname{Stab}(x)|=|G|$.  Sum orbit by orbit.
\end{proof}

Let $F_{\rm diag}$ be the signed number of one-factorizations fixed by a
nontrivial diagonal translation $(t,t)$.

\begin{lemma}[axis vanishing]\label{lem:axis-vanishing}
No one-factorization is fixed by a nontrivial translation on just one
layer.
\end{lemma}

\begin{proof}
Consider $(t,0)$ with $t\ne0$.  If an invariant perfect matching
contained a cross edge $\{(x,0),(y,1)\}$, its $p$ translates would all
contain the fixed vertex $(y,1)$, giving it degree $p$ in a matching,
which is impossible.  Hence every edge of an invariant matching would
lie within a layer.  But then each layer would be perfectly matched
within itself, and the layers have odd size $p$.  Thus there is no
invariant perfect matching at all.

On the other hand, if a one-factorization were fixed by $(t,0)$, then
the induced action on its $2p-1$ factors would have all orbits of size
$1$ or $p$, so the number of fixed factors would be congruent to
$2p-1\equiv-1\pmod{p}$; in particular some factor would be an invariant
perfect matching.  This contradiction proves the lemma; the other axis
is symmetric.
\end{proof}

Within the normalizer of $\Gamma$, every $(a,b)$ with $a,b\ne0$ is
conjugate to a diagonal translation: rescale the two layers independently
by $x\mapsto a^{-1}x$ and $y\mapsto b^{-1}y$.  By
Lemma~\ref{lem:sign-invariance}, all $(p-1)^2$ such elements have the
same signed trace $F_{\rm diag}$.  Lemma~\ref{lem:weighted-burnside}
together with Lemma~\ref{lem:axis-vanishing} gives
\[
                  S_{2p}+(p-1)^2F_{\rm diag}\equiv0\pmod{p^2}.
                                                               \tag{4.3}
\]
It remains to determine $F_{\rm diag}\pmod{p^2}$.

\subsection{Diagonal fixed points and cyclic starters}

For $d\in\F_p$, define the invariant cross matching
\[
 M_d=\bigl\{\{(x,0),(x+d,1)\}:x\in\F_p\bigr\}.            \tag{4.4}
\]
For a matching on an even subset of the ordered vertex set, define its
sign by counting crossings in the induced order; this sign is again
denoted by $\varepsilon$.
A \emph{starter} in $\F_p$ is a perfect matching $A$ of
$\F_p\setminus\{0\}$ whose $m=(p-1)/2$ undirected differences represent
all classes in $\F_p^\times/\{\pm1\}$, where the difference of
$\{a,b\}$ is the class $\{\pm(a-b)\}$.  Equivalently, the $p$
translates of $A$ partition the edges of $K_p$ into near-perfect
matchings.
For example, $\{\{1,4\},\{2,3\}\}$ is a starter in $\F_5$.

\begin{lemma}[classification]\label{lem:diag-classification}
Every diagonally invariant one-factorization of $K_{2p}$ is obtained
uniquely from the following data:
\begin{enumerate}[label=\textup{(\roman*)}]
\item a residue $d_0\in\F_p$;
\item the $p-1$ fixed matchings $M_d$ with $d\ne d_0$;
\item an ordered pair $(A,B)$ of starters;
\item the $p$ translates of the factor consisting of the cross
edge $\{(0,0),(d_0,1)\}$, the starter $A$ in $V_0$, and the translate
$B+d_0$ in $V_1$.
\end{enumerate}
\end{lemma}

\begin{proof}
Fix the diagonal translation $g=(t,t)$, $t\ne0$; conjugating by a
simultaneous rescaling we may take $t=1$.  First, $g$ fixes no edge: a
setwise fixed edge would either have both endpoints fixed, impossible
since $g$ has no fixed vertices, or be swapped by $g$, in which case
$g^2\ne\mathrm{id}$ would fix both endpoints, again impossible.  A
$g$-invariant perfect matching is therefore a union of full edge orbits,
and since it has exactly $p$ edges it is a single orbit.  The orbit of
an edge inside one layer consists of $p$ edges on the $p$ vertices of
that layer and is not a matching; the orbit of a cross edge is one of
$M_0,\ldots,M_{p-1}$.  Hence the $M_d$ are precisely the invariant
perfect matchings.

Now let $\mathcal F$ be a $g$-invariant one-factorization.  The number
of $g$-fixed factors is congruent to $2p-1\equiv-1\pmod{p}$, so it is
$p-1$ or $2p-1$; the latter would require $2p-1$ pairwise disjoint
invariant matchings, while only $p$ exist.  So exactly $p-1$ of the
$M_d$ occur as factors, one $M_{d_0}$ being omitted, and the remaining
$p$ factors form a single free orbit.  Their union is the complement of
the chosen fixed factors, namely the two layer cliques together with
$M_{d_0}$.  Each of these $p$ factors must use at least one cross edge,
because a factor with no cross edge would perfectly match each odd
layer; since only the $p$ edges of $M_{d_0}$ are available, each factor
contains exactly one of them, together with a near-perfect matching in
each layer.  Edge-disjointness of the $p$ translated factors is exactly
the condition that the layer parts be starters.  Normalizing at the
unique factor containing $\{(0,0),(d_0,1)\}$ makes the pair $(A,B)$
unique.
\end{proof}

Order each layer as $0,1,\ldots,p-1$.  For a starter $A$, let
$\varepsilon(A)$ be its Pfaffian sign as a matching of
$\F_p\setminus\{0\}$, and put
\[
                         T_p=\sum_{A\text{ starter}}\varepsilon(A).
                                                               \tag{4.5}
\]

\begin{lemma}[fixed-point sign]\label{lem:diag-sign}
The factorization associated with $(d_0,A,B)$ has sign
$\varepsilon(A)\varepsilon(B)$.  Consequently
\[
                         F_{\rm diag}=pT_p^2.              \tag{4.6}
\]
\end{lemma}

\begin{proof}
We compute the sign factor by factor.

\emph{Fixed factors.}  The matching $M_d$ consists of the $p$ arcs from
position $x$ in $V_0$ to position $p+((x+d)\bmod p)$ in $V_1$.  Two such
arcs, starting at $x<x'$, cross precisely when their right endpoints
appear in the same relative order, so
\[
 \crs(M_d)=\binom{p}{2}-\operatorname{inv}(\sigma_d),
\]
where $\operatorname{inv}$ denotes the number of inversions and
$\sigma_d$ is the cyclic shift by $d$.  The shift has
$d(p-d)$ inversions, where $d$ is represented in
$\{0,\ldots,p-1\}$.  This number is even since $p$ is odd, and
$\binom{p}{2}=pm\equiv m\pmod{2}$.  Hence $\varepsilon(M_d)=(-1)^m$ for
every $d$, and the $p-1$ fixed factors contribute
$(-1)^{m(p-1)}=1$ in total.

\emph{Rotation invariance.}  Two chords on a circle cross if and only
if their endpoints interleave cyclically, and cyclic interleaving is
unchanged by cutting the circle at any point.  Hence linear crossing
parity on the layer order $0,1,\ldots,p-1$ agrees with circular
crossing parity.  Therefore, for every $t\in\F_p$,
\[
 \varepsilon(A+t)=\varepsilon(A),\qquad
 \varepsilon(B+t)=\varepsilon(B).                         \tag{4.7}
\]

\emph{The free orbit.}  Consider the factor whose cross edge joins
$(t,0)$ to $(b_t,1)$, where $b_t=(t+d_0)\bmod p$; its layer parts are
$A+t$ (missing $t$) and $B+d_0+t$ (missing $b_t$).  Arcs inside $V_0$
and arcs inside $V_1$ occupy disjoint intervals and never cross each
other.  The cross edge starts at position $t$ inside $V_0$ and ends
inside $V_1$; it crosses exactly those chords of the layer parts that
straddle its endpoints.  In a near-perfect matching of
$\{0,\ldots,p-1\}$ missing $h$, the number of chords straddling $h$ has
the parity of $h$, since the $h$ vertices to the left of $h$ are matched
and those matched to the right are exactly the straddling chords.
Hence this factor has sign
\[
                         (-1)^{t+b_t}\,\varepsilon(A)\varepsilon(B),
\]
using (4.7).  Multiplying over $t=0,\ldots,p-1$ introduces no extra
sign, because
\[
 \sum_{t=0}^{p-1}t+\sum_{t=0}^{p-1}b_t=2\cdot\frac{p(p-1)}2=p(p-1)
\]
is even.  Since $p$ is odd, the free orbit contributes
$\bigl(\varepsilon(A)\varepsilon(B)\bigr)^p=\varepsilon(A)\varepsilon(B)$.

Summing over the $p$ choices of $d_0$ and the ordered pairs $(A,B)$
gives (4.6).
\end{proof}

\section{The starter sum}

Lemma~\ref{lem:diag-sign} reduces the remaining signed trace to
$pT_p^2$.  It remains to prove that $T_p$ is nonzero modulo $p$; a
skew-circulant cofactor gives its exact residue.

Let $m=(p-1)/2$ and introduce variables $y_1,\ldots,y_m$.  Let
$C=(c_{ij})_{i,j\in\F_p}$ be the skew-circulant matrix with symbol
\[
                         f(z)=\sum_{d=1}^m y_d(z^d-z^{-d}). \tag{5.1}
\]
Thus, for $i<j$ in the ordinary order,
\[
 c_{ij}=\begin{cases}
 y_{j-i},&j-i\le m,\\
 -y_{p-(j-i)},&j-i>m.
 \end{cases}                                               \tag{5.2}
\]
Let $B_p$ be the principal $(p-1)\times(p-1)$ minor obtained by deleting
the row and column indexed by $0$.

\begin{lemma}[skew-circulant cofactor]\label{lem:skew-circulant}
Over $\F_p[y_1,\ldots,y_m]$,
\[
                         \Pf(B_p)=\left(\frac{2}{p}\right)
 \left(2\sum_{d=1}^m d\,y_d\right)^{\!m}                   \tag{5.3}
\]
where $\left(\frac{\cdot}{p}\right)$ denotes the Legendre symbol.
\end{lemma}

\begin{proof}
Regard the integer polynomial matrix $C$ as a matrix over
$\mathbb Q(\omega)[y_1,\ldots,y_m]$, where $\omega$ is a primitive
$p$th root of unity.  Its eigenvalues are $f(\omega^k)$,
$0\le k<p$, and $f(1)=0$.  Generically its kernel is exactly
$\langle\mathbf{1}\rangle$: after the specialization $y_1=1$ and
$y_d=0$ for $d>1$, one has
$f(\omega^k)=\omega^k-\omega^{-k}\ne0$ for $k\ne0$.  On this generic
locus, $C\cdot\operatorname{adj}C=(\det C)I=0$, so every column of
$\operatorname{adj}C$ lies in this kernel.  Since
$\operatorname{adj}C$ is again circulant, it equals $cJ$, where $J$ is
the all-ones matrix and $c$ is a scalar.  Taking traces gives
\[
 pc=\operatorname{tr}(\operatorname{adj}C)
    =\prod_{k=1}^{p-1}f(\omega^k),
\]
the elementary symmetric function of the nonzero eigenvalues.  Every
principal cofactor of $C$ is therefore
\[
                         \det B_p=\frac1p\prod_{k=1}^{p-1}f(\omega^k).
                                                               \tag{5.4}
\]
Both sides of (5.4) are polynomial functions of the variables $y_d$,
so the identity extends from the generic locus to every specialization.
Put $g(z)=f(z)/(z-1)$, a Laurent polynomial since $f(1)=0$.  Since
$\prod_{k=1}^{p-1}(\omega^k-1)=p$, (5.4) becomes
\[
                         \det B_p=\prod_{k=1}^{p-1}g(\omega^k).  \tag{5.5}
\]
Both sides have integer coefficients as polynomials in the $y_d$.
Reduce coefficientwise modulo the cyclotomic prime $(1-\omega)$ lying
above $p$.  The quotient is
$\mathbb Z[\omega]/(1-\omega)\cong\F_p$, so every $\omega^k$ becomes
$1$, while
\[
                         g(1)=f'(1)=2\sum_{d=1}^m d\,y_d.
\]
Thus
\[
 \det B_p=\Pf(B_p)^2
       =\left(2\sum_{d=1}^m d\,y_d\right)^{\!p-1}
       \quad\text{in }\F_p[y_1,\ldots,y_m].               \tag{5.6}
\]
This polynomial ring is an integral domain and $p$ is odd, so the two
square roots in (5.6) differ by a sign independent of the variables.
It remains to determine that sign.  Set $y_1=1$ and $y_d=0$ for
$d>1$.  Then $B_p$ is the skew adjacency matrix of the path
$1,2,\ldots,p-1$, with $+1$ immediately above the diagonal.  Its only
perfect matching is
$\{1,2\},\{3,4\},\ldots,\{p-2,p-1\}$, and hence its Pfaffian is $1$.
The polynomial $\bigl(2\sum d\,y_d\bigr)^m$ specializes to $2^m$.
Euler's criterion gives $2^m\equiv\left(\frac{2}{p}\right)\pmod p$;
since this value is its own inverse, the constant sign in (5.6) is
$\left(\frac{2}{p}\right)$.
\end{proof}

\begin{lemma}[starter sum]\label{lem:starter-sum}
\[
                         T_p\equiv-\left(\frac{2}{p}\right)\pmod{p}.
                                                               \tag{5.7}
\]
In particular, $T_p^2\equiv1\pmod p$.
\end{lemma}

\begin{proof}
In the Pfaffian expansion of $B_p$, the coefficient of the squarefree
monomial $y_1\cdots y_m$ is a signed sum over the matchings of
$\{1,\ldots,p-1\}$ using each difference class exactly once, that is,
over starters.  An edge $\{a,b\}$ with $1\le a<b\le p-1$ carries the
extra minus sign in (5.2) exactly when $b-a>m$; call such an edge long,
and let $L(A)$ be the number of long edges of the starter $A$.  The
parity of $L(A)$ does not depend on $A$: the ordinary differences of
the matched pairs sum modulo $2$ to
\[
 \sum_{\text{edges}}(b-a)\equiv\sum_{x=1}^{2m}x\equiv m\pmod{2},
\]
while each short edge contributes its cyclic difference $d$ and each
long edge contributes $p-d\equiv d+1\pmod{2}$.  Since the cyclic
differences of a starter are exactly $1,\ldots,m$,
\[
 L(A)\equiv m-\sum_{d=1}^m d
       \equiv\frac{m(m-1)}2\pmod{2}.                         \tag{5.8}
\]
Thus $[y_1\cdots y_m]\Pf(B_p)=(-1)^{m(m-1)/2}\,T_p$.

Formula (5.3), on the other hand, gives this coefficient as
\[
 \left(\frac{2}{p}\right)m!\prod_{d=1}^m(2d)
 =\left(\frac{2}{p}\right)2^m(m!)^2.                       \tag{5.9}
\]
By Euler's criterion, the product of the first two factors on the
right of (5.9) is $1$.  Wilson's theorem gives
$(m!)^2\equiv(-1)^{m+1}\pmod p$, and comparison with (5.8) therefore
yields
\[
 T_p\equiv(-1)^{m(m+1)/2+1}\pmod p.
\]
Finally,
$\left(\frac{2}{p}\right)=(-1)^{(p^2-1)/8}
=(-1)^{m(m+1)/2}$, proving (5.7).
\end{proof}

Combining Lemma~\ref{lem:diag-sign} with
Lemma~\ref{lem:starter-sum} yields
\[
                         F_{\rm diag}\equiv p\pmod{p^2}. \tag{5.10}
\]

\begin{proof}[Proof of Theorem~\ref{thm:intro-2p}]
Insert (5.10) into the Burnside congruence (4.3):
\[
 S_{2p}\equiv-(p-1)^2F_{\rm diag}
          \equiv-(p-1)^2p
          \equiv-p\pmod{p^2}.
\]
This residue is nonzero, so $S_{2p}$ is a nonzero integer, and
Proposition~\ref{prop:coefficient-bridge} gives
\[
 \chi'_{\mathrm P}(K_{2p})=\chi'_{\ell}(K_{2p})=2p-1.
\]
\end{proof}

\subsection{The three prime-indexed orders}

The starter residue also recovers the previously known prime-degree
family and places the three prime-indexed orders in one calculation.

\begin{proposition}[the prime-degree family from starters]
\label{prop:pplus1-starter}
For every odd prime $p$,
\[
 S_{p+1}\equiv(-1)^{(p-1)/2}T_p
       \equiv-\left(\frac{-2}{p}\right)\pmod p.            \tag{5.11}
\]
Consequently,
\[
 \chi'_{\mathrm P}(K_{p+1})
   =\chi'_{\ell}(K_{p+1})=p.
\]
\end{proposition}

\begin{proof}
Put $m=(p-1)/2$.
Write the vertices as $\F_p\cup\{\infty\}$, in that order, and let
$\F_p$ translate its own vertices while fixing $\infty$.  The signed
Burnside congruence and conjugacy of the nonzero translations give
\[
                         S_{p+1}+(p-1)F_+\equiv0\pmod p,   \tag{5.12}
\]
where $F_+$ is the signed trace of $x\mapsto x+1$.

No perfect matching is fixed by this translation: the partner of
$\infty$ would have to be another fixed vertex.  Hence a fixed
one-factorization, which has $p$ factors, consists of one free orbit of
factors.  In that orbit there is a unique factor containing
$\{\infty,0\}$; it has the form
\[
                         \{\infty,0\}\cup A,
\]
where $A$ is a starter on $\F_p\setminus\{0\}$.  Conversely, the
translates of this matching form a one-factorization precisely when
$A$ is a starter.  Thus fixed one-factorizations are in bijection with
starters.

For $0\le t<p$, the factor
$\{\infty,t\}\cup(A+t)$ has sign $(-1)^t\varepsilon(A)$.  Indeed,
(4.7) gives $\varepsilon(A+t)=\varepsilon(A)$, and the arc
$\{t,\infty\}$ crosses a number of chords of $A+t$ congruent to $t$
modulo $2$.  Multiplication over the $p$ translates gives
\[
 \prod_{t=0}^{p-1}\varepsilon\bigl(\{\infty,t\}\cup(A+t)\bigr)
       =(-1)^{p(p-1)/2}\varepsilon(A)^p
       =(-1)^m\varepsilon(A).
\]
Therefore $F_+=(-1)^mT_p$.  Equation (5.12) says
$S_{p+1}\equiv F_+\pmod p$, and Lemma~\ref{lem:starter-sum} gives
\[
 (-1)^mT_p
 =-(-1)^m\left(\frac{2}{p}\right)
 =-\left(\frac{-2}{p}\right).
\]
The residue is nonzero, so Proposition~\ref{prop:coefficient-bridge}
gives the paintability statement.
\end{proof}

\begin{corollary}[three residues, one starter sum]
\label{cor:three-residues}
For every odd prime $p$, put $m=(p-1)/2$.  Then
\begin{align*}
 S_{p-1}&\equiv(-1)^{m+1}T_p
          \equiv \left(\frac{-2}{p}\right) &&\pmod p,\\
 S_{p+1}&\equiv(-1)^mT_p
          \equiv-\left(\frac{-2}{p}\right) &&\pmod p,\\
 S_{2p}&\equiv-pT_p^2\equiv-p &&\pmod {p^2}.
\end{align*}
In particular,
\[
 S_{p-1}+S_{p+1}\equiv0\pmod p,
 \qquad
 S_{p-1}S_{p+1}\equiv\frac{S_{2p}}p\equiv-1\pmod p.       \tag{5.13}
\]
\end{corollary}

\begin{proof}
The first line follows from Theorem~\ref{thm:intro-pminus1}, since
$(p-2)!\equiv1\pmod p$, together with
Lemma~\ref{lem:starter-sum}.  The second is
Proposition~\ref{prop:pplus1-starter}.  For the third, (4.3) and
Lemma~\ref{lem:diag-sign} give
$S_{2p}\equiv-pT_p^2\pmod {p^2}$.  The final relations follow at once.
\end{proof}

\section{Deleting factors}

Identity (6.1) relates the signed counts before and after one factor is
deleted.  The second result in this section evaluates the two-layer
count after cyclic cross matchings are deleted.

\subsection{Deleting one factor}

\begin{proposition}[deleting one factor]\label{prop:matching-deletion}
Let $G$ be a simple $k$-regular one-factorable graph on an ordered vertex set,
and let $\mathcal M(G)$ be its set of perfect matchings.  Then
\[
 \sum_{M\in\mathcal M(G)}\varepsilon(M)S(G-M)=kS(G).       \tag{6.1}
\]
In particular, if $n\ge4$ is even and $I$ is any perfect matching of
$K_n$, then
\[
 S(K_n-I)=\varepsilon(I)\frac{(n-1)S_n}{(n-1)!!}
          =\varepsilon(I)\frac{S_n}{(n-3)!!}.             \tag{6.2}
\]
Consequently, $(n-3)!!$ divides $S_n$, and
\[
 S_n\ne0
 \quad\Longrightarrow\quad
 \chi'_{\mathrm P}(K_n-I)=\chi'_{\ell}(K_n-I)=n-2.       \tag{6.3}
\]
\end{proposition}

\begin{proof}
Expand the left side of (6.1).  A term consists of a
one-factorization $\mathcal F$ of $G$ together with one distinguished
factor $M\in\mathcal F$; its sign is
\[
 \varepsilon(M)\varepsilon(\mathcal F\setminus\{M\})
 =\varepsilon(\mathcal F).
\]
Every factorization has $k$ choices for $M$, proving (6.1).

Now take $G=K_n$.  Adjoining $I$ to a factorization of $K_n-I$ gives
\[
 \varepsilon(I)S(K_n-I)
   =\sum_{\substack{\mathcal F\text{ a one-factorization of }K_n\\
                    I\in\mathcal F}}\varepsilon(\mathcal F).
\]
By Lemma~\ref{lem:sign-invariance}, this quantity is independent of
$I$, since any two perfect matchings are related by a relabeling.
There are $(n-1)!!$
perfect matchings, so (6.1) gives
\[
 (n-1)!!\,\varepsilon(I)S(K_n-I)=(n-1)S_n.
\]
This is (6.2), since $\varepsilon(I)^2=1$.  The divisibility and
nonvanishing assertions follow immediately.  Finally, $K_n-I$ is
$(n-2)$-regular and one-factorable; apply
Proposition~\ref{prop:coefficient-bridge}.
\end{proof}

\subsection{Cyclic cross-factor deletion}

The same diagonal trace evaluates the signed count when any nonempty
proper collection of its invariant cross matchings is retained.

Retain the two ordered layers $V_0,V_1$ and the matchings $M_d$
introduced in Section~4.  For a nonempty set $D\subseteq\F_p$, put
\[
 G_D=(K_p\sqcup K_p)\cup\bigcup_{d\in D}M_d,
 \qquad \delta=|D|,                                      \tag{6.4}
\]
and let $\tau(x,i)=(x+1,i)$ be the diagonal translation.  Write
$T_D$ for the signed sum of the one-factorizations of $G_D$ fixed by
$\tau$.

\begin{theorem}[two-layer trace]\label{thm:two-layer-trace}
Let $p$ be an odd prime, put $m=(p-1)/2$, and let
$\varnothing\ne D\subseteq\F_p$.
\[
             T_D=(-1)^{m(\delta-1)}\delta T_p^2.          \tag{6.5}
\]
If $1\le\delta\le p-1$, then
\[
 S(G_D)\equiv(-1)^{m(\delta-1)}\delta\pmod p,            \tag{6.6}
\]
and consequently
\[
 \chi'_{\mathrm P}(G_D)=\chi'_{\ell}(G_D)=p-1+\delta.   \tag{6.7}
\]
\end{theorem}

\begin{proof}
The automorphism $\tau$ is a product of two $p$-cycles, so
$\det(P_\tau)=1$.  It reverses $p-1$ internal edges in each layer and
no cross edge in the chosen vertex order.  Hence
$\iota_{G_D}(\tau)=2(p-1)$, and Lemma~\ref{lem:sign-invariance} shows
that $\tau$ preserves the sign of every one-factorization.

The $\tau$-invariant perfect matchings of $G_D$ are exactly the
$M_d$ with $d\in D$.  Indeed, an invariant matching is a single
$p$-edge orbit; an internal edge orbit is $2$-regular on one odd
layer, whereas a cross-edge orbit is some $M_d$.

Suppose that a $\tau$-fixed one-factorization has $a$ fixed factors.
Its remaining factors occur in orbits of length $p$, and its degree is
$p-1+\delta$.  Thus
\[
                       a\equiv\delta-1\pmod p.
\]
Since $0\le a\le\delta\le p$, this forces $a=\delta-1$.
Exactly one available cross matching, say $M_{d_0}$, is therefore
omitted from the fixed factors, and the remaining $p$ factors form one
free orbit.  Their union is
\[
                       K_p\sqcup K_p\cup M_{d_0}.
\]
The free orbit has the same union and the same diagonal action as in
Lemma~\ref{lem:diag-classification}.  After marking its unique factor
containing $\{(0,0),(d_0,1)\}$, that lemma identifies the two internal
parts with an ordered pair $(A,B)$ of starters.

Each of the $\delta-1$ fixed matchings has sign $(-1)^m$, and the free
orbit has sign $\varepsilon(A)\varepsilon(B)$ by the calculation in
Lemma~\ref{lem:diag-sign}.  Summing over the $\delta$ choices of
$d_0$ and over the ordered pairs of starters gives (6.5).

For $\delta\le p-1$, weighted Burnside for $\langle\tau\rangle$ gives
\[
 S(G_D)+(p-1)T_D\equiv0\pmod p,
\]
because every nonidentity power of $\tau$ has the same fixed set.
Thus $S(G_D)\equiv T_D\pmod p$.  Lemma~\ref{lem:starter-sum} gives
$T_p^2\equiv1\pmod p$, proving (6.6).  The displayed residue is
nonzero.  The same lemma implies $T_p\ne0$, so a starter exists; the
fixed factorizations constructed above therefore show that $G_D$ is
class $1$.  Proposition~\ref{prop:coefficient-bridge} proves (6.7).
\end{proof}

At $\delta=p$, (6.5) becomes $F_{\rm diag}=pT_p^2$, so the calculation
for $K_{2p}$ must be carried out modulo $p^2$.  When
$1\le\delta<p$, the trace is nonzero modulo $p$.

The deleted form makes the coloring consequence especially transparent.

\begin{corollary}[cyclic cross-factor deletions]
\label{cor:cyclic-deletions}
Let $p$ be an odd prime, put $m=(p-1)/2$, let $E\subset\F_p$ with
$1\le r=|E|\le p-1$, and put
\[
                         H_E=K_{2p}-\bigcup_{e\in E}M_e.
\]
Then
\[
 S(H_E)\equiv(-1)^{mr+1}r\pmod p,                        \tag{6.8}
\]
and
\[
 \chi'_{\mathrm P}(H_E)=\chi'_{\ell}(H_E)=2p-1-r.       \tag{6.9}
\]
\end{corollary}

\begin{proof}
Apply Theorem~\ref{thm:two-layer-trace} to
$D=\F_p\setminus E$, for which $\delta=p-r$.  Since
\[
 m(p-r-1)\equiv mr\pmod2,
 \qquad p-r\equiv-r\pmod p,
\]
equation (6.6) becomes (6.8), and (6.9) follows from (6.7).
\end{proof}

At the endpoints, $r=p-1$ gives
$H_E\cong K_p\square K_2$, while $r=1$ gives the cocktail-party graph
$K_{2p}-I$.  For $r=2$, the union of the two deleted matchings is a Hamiltonian
cycle, so the same theorem applies to $K_{2p}-C_{2p}$.  The case
$r=1$ supplies the cyclic representative; Proposition~\ref{prop:matching-deletion}
gives the exact quotient formula for every perfect matching $I$.

\subsection{Coloring consequences}

Applying the deletion identity (6.2) to the three congruences in
Corollary~\ref{cor:three-residues} gives the following consequences.

\begin{corollary}[prime-indexed one-factor deletions]
\label{cor:prime-deletions}
Let $I$ be a perfect matching.  For every odd prime $p$,
\[
 \chi'_{\mathrm P}(K_{p+1}-I)
 =\chi'_{\ell}(K_{p+1}-I)=p-1
\]
and
\[
 \chi'_{\mathrm P}(K_{2p}-I)
 =\chi'_{\ell}(K_{2p}-I)=2p-2.
\]
For every odd prime $p\ge5$,
\[
 \chi'_{\mathrm P}(K_{p-1}-I)
 =\chi'_{\ell}(K_{p-1}-I)=p-3.
\]
Moreover,
\[
             \varepsilon(I)S(K_{2p}-I)\equiv-1\pmod p.     \tag{6.10}
\]
\end{corollary}

\begin{proof}
The three signed counts $S_{p-1}$, $S_{p+1}$, and $S_{2p}$ are
nonzero by Corollary~\ref{cor:three-residues}.  The coloring statements
therefore follow from Proposition~\ref{prop:matching-deletion}.  For
the last congruence, write
\[
 (2p-3)!!=p\,d.
\]
Modulo $p$, the factors before $p$ give the odd residues
$1,3,\ldots,p-2$, while those after $p$ reduce to the even residues
$2,4,\ldots,p-3$.  Hence $d\equiv(p-2)!\equiv1\pmod p$.  Dividing
$S_{2p}\equiv-p\pmod {p^2}$ by $p\,d$ in (6.2) proves (6.10).
\end{proof}

\section*{Conclusion}
\addcontentsline{toc}{section}{Conclusion}

The two new congruences come from different calculations, but their residues
are governed by the same signed starter sum
\[
                     T_p\equiv-\left(\frac2p\right)\pmod p.
\]
The same calculation also recovers the known prime-degree family and gives
\[
 \begin{aligned}
 S_{p-1}&\equiv \left(\frac{-2}{p}\right) &&\pmod p,\\
 S_{p+1}&\equiv-\left(\frac{-2}{p}\right) &&\pmod p,\\
 S_{2p}&\equiv-p &&\pmod {p^2}.
 \end{aligned}
\]
These nonzero residues prove $\chi'_{\mathrm P}=\chi'_{\ell}=\chi'$ at the
two new orders $p-1$ and $2p$.  They also give the one-factor-deletion
results at all three orders, while the same starter sum controls the proper
cyclic cross-factor deletions at order $2p$.  Thus one signed matching sum
accounts for all three prime-indexed orders and for their deletion
consequences.

The companion paper describes the boundary of the prime-block translation method
\cite{JafariSaturation}.  For an odd prime $p$, write the even integer $n$ as
$n=bp+h$, where $b\ge1$ and $0\le h<p$.  The companion paper lets $\F_p^b$
translate $b$ disjoint $p$-vertex blocks.  Except when
$(b,h)=(1,1)$ or $(2,0)$, every nonidentity signed fixed-point trace is
divisible by the full group order $p^b$.  Weighted Burnside then gives
$p^b\mid S_n$, rather than a nonzero residue.  Thus, by itself, the same
congruence argument yields no further coloring cases beyond the two
exceptional translation configurations; the order $p-1$ result comes
instead from the separate Pfaffian calculation of Section~3.

\clearpage
\phantomsection

\end{document}